\documentclass[reqno,11pt]{amsart}

\usepackage{amssymb}
\usepackage{amsgen}
\usepackage{amsmath}
\usepackage{amsthm}
\usepackage{mathrsfs}
\usepackage{cite}
\usepackage{amsfonts}
\usepackage{tikz}

\newcommand{\qv}{\mathop{\mathrm{qv}}}

\renewcommand{\to}{\longrightarrow}

\newcommand{\sgn}{\mathop{\mathbf s}\nolimits}

\usepackage{xcolor}

\def\AA{\mathcal A} 
\def\FF{\mathcal F} 

\newtheorem{Thm}{Theorem}
\newtheorem{Prop}[Thm]{Proposition}

{\theoremstyle{definition}
}
{\theoremstyle{remark}
\newtheorem{Rmk}[Thm]{Remark}}
\newtheorem{Cor}[Thm]{Corollary}
{\theoremstyle{remark}
}

\newtheorem{Problem}[Thm]{Problem}

{\theoremstyle{remark}
}
{\theoremstyle{remark}
}
{\theoremstyle{remark}
}

\numberwithin{equation}{section}

\title{Semigroups embeddable in hyperplane face monoids}

\author{Stuart Margolis}
\address[S.~Margolis]{%
    Department of Mathematics\\
    Bar Ilan University\\
    52900 Ramat Gan\\
    Israel}
\email{margolis@math.biu.ac.il}

\author{Franco Saliola}
\address[F.~~Saliola]{
D{\'e}partement de Math{\'e}matiques -- LaCIM\\
Universit{\'e} du Qu{\'e}bec {\`a} Montr{\'e}al\\
C.P. 8888, Succursale Centre-Ville\\
Montr{\'e}al, Qu{\'e}bec  H3C 3P8\\
Canada
}
\email{saliola.franco@uqam.ca}

\author{Benjamin Steinberg}
\address[B.~Steinberg]{%
    Department of Mathematics\\
    City College of New York\\
    Convent Avenue at 138th Street\\
    New York, New York 10031\\
    USA}
\email{bsteinberg@ccny.cuny.edu}

\thanks{The second was supported in part by NSERC. The third author was supported in part by the Simon's Foundation collaboration grant 245268}
\date{\today}


\begin{document}

\begin{abstract}
The left regular band structure on a hyperplane arrangement and its representation theory provide an important connection between semigroup theory and algebraic combinatorics. A finite semigroup embeds in a real hyperplane face monoid if and only if it is in the quasivariety generated by the monoid obtained by adjoining an identity to the two-element left zero semigroup.  We prove that this quasivariety is on the one hand polynomial time decidable, and on the other minimally non-finitely based.  A similar result is obtained for the semigroups embeddable in complex hyperplane semigroups.
\end{abstract}

\maketitle

\section{Introduction}
A left regular band is a semigroup satisfying the identities $x^2=x$ and $xyx=xy$.  In recent years, there has been a great deal of interest in left regular bands because a number of combinatorial structures have been observed to naturally carry the structure of a left regular band including: real hyperplane arrangements and oriented matroids~\cite{BHR,Brown1,Brown2,Saliolahyperplane,DiaconisBrown1}, complex hyperplane arrangments~\cite{bjorner2}, matroids~\cite{Brown1}, interval greedoids~\cite{bjorner2} and oriented interval greedoids~\cite{SaliolaThomas};   see also~\cite{GrahamChung}.  Further applications of left regular bands and their representation theory to algebraic combinatorics and probability theory can be found in~\cite{Aguiar,aguiarspecies,Brown1,Brown2,Saliolahyperplane,Saliola,lrbglobal}.  Left regular bands also appear in the work of Lawvere on topos theory~\cite{Lawvere,graphic1,moregraphic}. Thus left regular bands have proven to be one of the most important classes of semigroups from the point of view of applications to other areas of mathematics.

It turns out that many of the left regular bands associated to combinatorial structures are subsemigroups of direct powers of the monoid $L=\{0,+,-\}$ consisting of an identity $0$ and two left zeroes $\{+,-\}$. We will explain the use of the symbols $\{+,-\}$ for the left zeroes in the next section. For example,  $\{0,+,-\}^n$ is the face monoid~\cite{Brown1,Saliolahyperplane} of the hyperplane arrangement consisting of the coordinate hyperplanes in $\mathbb R^n$.  In this paper we characterize those semigroups embeddable in direct powers of $L$.

It is well known, cf.~\cite[Proposition~7.3.2]{qtheor}, that the variety of left regular bands is generated by $L$. The only proper subvarieties of left regular bands are the trivial variety, the variety of left zero semigroups, the variety of semilattices and the  variety of left normal bands~\cite{Bir, Fenn}. It was thus a huge surprise when Mark Sapir exhibited uncountably many quasivarieties of left regular bands~\cite{Sapirquasi}.   Moreover, he produced a countable anti-chain in the lattice of left regular band quasivarieties, each of whose members is generated by a finite left regular band~\cite{Sapirquasi}.

It is then entirely natural from the point of view of semigroup theory to ask what is the quasivariety generated by $L$ and whether it is finitely based.
Recall that a \emph{quasi-identity} is a universally quantified formal implication of the form \[(u_1=v_1)\wedge \cdots \wedge (u_n=v_n)\implies u=v\]  where $u_i,v_i,u,v$ are words over some alphabet, for $1\leq i\leq n$.  For example, the left cancellation law is the quasi-identity $xy=xz\implies y=z$.  There is an obvious notion of what it means for a semigroup to satisfy a quasi-identity. For instance, a semigroup is left cancellative if and only if it satisfies the left cancellation law.

A class of semigroups defined by quasi-identities is called a \emph{quasivariety}~\cite{universalalgebra}.  Quasivarieties alternatively can be described as classes of semigroups closed under isomorphism, subsemigroups, direct products and ultraproducts~\cite{universalalgebra}.  If $S$ is a finite semigroup, then the quasivariety $\qv(S)$ generated by $S$ consists of all semigroups embeddable in a direct power of $S$, or equivalently of all semigroups whose homomorphisms to $S$ separate points.  In particular, if $T$ is a finite semigroup we have the following  exponential time algorithm to determine if $T$ belongs to $\qv(S)$: compute all $|S|^{|T|}$ maps from $T$ to $S$; check which are homomorphisms and check if they suffice to separate points. It is well known that this can be turned into a non-deterministic polynomial time ($\mathsf{NP}$) algorithm to decide membership in $\qv(S)$~\cite{Bergman}. If $\qv(S)$ has a finite basis of quasi-identities, then clearly membership in $\qv(S)$ for finite semigroups can be decided in polynomial time.

It is known that $\qv(L)$ contains finitely many proper subquasivarieties, all of which are finitely based~\cite{Gerhard}.
Let $ZL$ be the result of adjoining a multiplicative zero element to $L$.  Our main result is the following.

\begin{Thm}\label{main}
The quasivarieties $\qv(L)$ and $\qv(ZL)$ are not finitely based, but have polynomial time membership algorithms.
\end{Thm}

More precise descriptions of these quasivarieties will be given in the body of the article.

\begin{Rmk}
As noted above, the membership problem for the quasivariety generated by a finite semigroup $S$ (or more generally a finite universal algebra) is in the complexity class $\mathsf{NP}$. See~\cite{Bergman} for example. Jackson and McKenzie~\cite{Jackson} showed, among other related results, that there is a finite semigroup $S$ such that the membership problem for $\qv(S)$ is $\mathsf{NP}$-complete. Thus Theorem~\ref{main} shows that having a finite basis of quasi-identities is not the only way that a finite semigroup can have a polynomial time algorithm for membership in the quasivariety it generates. This is not the first example of such a phenomenon, but we emphasize that it does show that the language of quasi-identities is not powerful enough to capture finitely generated quasivarieties with polynomial time membership. The search for such a formalism is a worthy problem.
\end{Rmk}

\section{Historical background}\label{history}

The finite basis problems for the identity theory and quasi-identity theory of varieties and quasivarieties of semigroups, and more generally for universal algebras, have a long and extensive history. In this section we provide background on this work, mainly for semigroups, to put the results of the current paper in context.

Lyndon~\cite{Lyndon} proved that every universal algebra on a two-element set has a finite basis for its identities. Shafaat~\cite{Shafaat} proved that the lattice of quasivarieties of any universal algebra on a two-element set is a two-element chain. It follows that the quasivariety generated by a two-element universal algebra is a variety and thus has a finite basis for its quasi-identities.

For three-element universal algebras and three-element semigroups the situation is radically different. Gorbunov~\cite{Gorbunov} constructed a universal algebra with two unary operations possessing no independent basis of quasi-identities. Sapir~\cite{Sapir3} proved that the three-element monogenic aperiodic semigroup $C_{3,1}=\{x,x^{2}, x^{3}=x^{4}\}$ has no independent basis of quasi-identities. Thus the main result of this paper gives another example of a three-element semigroup with no finite basis of identities. The question of whether $L$ has an independent basis of identities is still open. See the last section of the paper on open problems.

Important classes of semigroups have also been studied with respect to the question of finite bases of identities and quasi-identities. Oates and Powell~\cite{OP}  proved that every finite group has a finite basis for its identities. Olshanskii~\cite{Olsh} proved that a finite group $G$ has a finite basis of quasi-identities if and only if every nilpotent subgroup of $G$ is abelian. Sapir~\cite{Sapir3} extended this latter result to the class of completely simple semigroups. A completely simple semigroup $S$ has a finite basis of quasi-identities if and only if $S$ is a rectangular group, which by definition is the direct product of a group and a rectangular band (equivalently $S$ has a Rees matrix representation in which the structure matrix has all entries the identity element of its maximal subgroup), such that every nilpotent subgroup of $S$ is abelian.

The most general results about finite semigroups with respect to finite bases of quasi-identities appeared in~\cite{JV}. We recall some definitions. The monogenic semigroup with index $i$ and period $p$ is $C_{i,p}$, the semigroup with presentation $\langle x\mid x^i =x^{i+p}\rangle$. The index of a finite semigroup is the maximum index of any of its monogenic semigroups. A semigroup $S$ is \emph{proper $3$-nilpotent} if it has a zero $0$, is not a semigroup with all products equal to $0$ and $S^3=0$. Jackson and Volkov~\cite{JV} prove that if $S$ is any finite semigroup that generates a quasivariety that contains either a proper 3-nilpotent semigroup or has index at least $3$, then $S$ has no finite basis for its quasi-identities. This includes Sapir's results~\cite{Sapir3} about the aperiodic monogenic semigroup of order (and index) $3$. The main proof techniques of~\cite{JV} reduce the general case to this one.

It is still a major open problem as to whether it is decidable if a finite semigroup has a finite basis of either its identities or quasi-identities. For the identity theory of finite semigroups, the notion of inherently non-finitely based (INFB) semigroups, has proved extremely useful. A finite semigroup $S$ is INFB if every locally finite variety $V$ containing $S$ is not finitely based. A well known example of an INFB is the six element semigroup consisting of all the $2 \times 2$ matrix units, the identity matrix and the zero matrix. While there has been much progress on the finite basis problem for finite semigroups, the general problem remains open. See~\cite{Volk} for an extensive survey of this area. We remark that for universal algebras, the problem is undecidable~\cite{McKen}.

For quasivarieties, there is the analogous notion of inherently non-finitely quasi-identity based (INFQB). There are examples of INFQB universal algebras~\cite{LawWill}, but no finite semigroup is INFQB~\cite{MargSap}. A related property is that of being strongly non-finitely quasi-identity based (SNFQB). A universal algebra $A$ is SNFQB if whenever $A$ is contained in a quasivariety $Q$ generated by a finite number of finite algebras, then $Q$ is not finitely based. Thus the results of~\cite{Sapir3} and~\cite{JV} say that the monogenic semigroup $C_{i,p}$ is SNFQB as long as $i>2$ as well as any proper 3-nilpotent semigroup. The results of~\cite{Olsh} and~\cite{Sapir3} imply that a completely simple semigroup $S$ is SNFQB if and only if $S$ either contains a non-abelian nilpotent subgroup or the idempotents of $S$ are not a subsemigroup. The paper~\cite{JV} discusses a general framework in which the properties INFQB and SNFQB are special cases.

Thus the question of whether it is decidable if a finite semigroup has a finite basis of quasi-identities has been reduced to the case that the semigroup has index at most $2$. Finite semigroups of index $1$ are exactly the class of completely regular semigroups, that is, semigroups that are unions of their subgroups. Aperiodic completely regular semigroups are exactly bands. Gerhard and Shafaat~\cite{Gerhard} proved that any normal band has a finite basis of quasi-identities. Recall that a band is normal if it satisfies the identity $xyzw=xzyw$. It is well known that a band $S$ is normal if and only if it is locally a semilattice, that is, $eSe$ is a semilattice for every $e \in S$. This in turn is equivalent to the condition that neither the left regular band $L$ nor its dual $R$, consisting of two right zeroes and an identity, is a subsemigroup of $S$.

Thus the main result Theorem~\ref{main} of this paper shows that $L$ and the dual semigroup $R$ are the smallest bands generating a non-finitely based quasivariety. Therefore, the varieties of left regular bands and right regular bands are the smallest varieties of bands that contain a non-finitely based quasivariety. It is an open question whether $L$ and $R$ are SNFQB within the class of bands or all finite semigroups. On the other hand, in an unpublished part of his thesis~\cite{SapirThesis}, Mark Sapir proved that the variety \mbox{$[\![x^{2}=x, xyz = xyzxz]\!]$} contains finite semigroups $S$ that are SNFQB. This latter variety of bands contains, but is strictly bigger than, the variety of left regular bands. As far as we know, prior to this work there were no examples of a left regular band without a finite basis. It is reasonable to ask if $L$ is SNFQB with respect to left regular bands,  bands, or all finite semigroups. Neither the techniques of Sapir's thesis~\cite{SapirThesis} nor the techniques of this paper shed light on this problem as far as we see. See the last section of this paper for a further discussion of this problem.

\section{Hyperplane arrangements and their face monoids}

We gather here the basic definitions, properties and examples of hyperplane arrangements and hyperplane face monoids. See~\cite{BHR,Brown1,Brown2,Saliolahyperplane,DiaconisBrown1} for further details and examples.

A \emph{hyperplane arrangement} $\AA$ in $\mathbb R^d$ is a finite set of
hyperplanes in $\mathbb R^d$. We restrict our attention to \emph{central}
hyperplane arrangements, that is arrangements where all the hyperplanes contain $0 \in \mathbb
R^d$. Each hyperplane $H \in \AA$ determines two open half-spaces of $\mathbb
R^d$ denoted $H^+$ and $H^-$. The choice of which half-space to label $+$
or $-$ is arbitrary, but fixed.

A \emph{face} of $\AA$ is a nonempty intersection of the form
\begin{gather*}
F = \bigcap_{H\in \AA} H^{\sigma_H},
\end{gather*}
where $\sigma_H \in \{+,-,0\}$ and $H^0 = H$. The sequence $\sigma(F) =
(\sigma_H)_{H\in\AA}$ is the \emph{sign sequence} of $F$. A \emph{chamber}
$C$ is a face such that $\sigma_H(C) \neq 0$ for all $H \in \AA$.

The next figure contains an example of a hyperplane arrangement with its sign sequence.  The hyperplanes are given by the $x$-axis and the lines at angles $2\pi/3$ and $4\pi/3$ with respect to the $x$-axis.

\begin{figure}[htb]
\centering
\begin{tikzpicture}
    \draw[line width=1.5pt] (-2,-2) -- (2,2);
    \draw[line width=1.5pt] (2,-2) -- (-2,2);
    \draw[line width=1.5pt] (-3,0) -- (3,0);
    \draw (0,0)       node[fill=white,font=\tiny] {$(000)$};
    \draw (0,-1.5)    node[font=\small]           {$(-+-)$};
    \draw (-1.4,-1.4)       node[fill=white,font=\tiny] {$(-+0)$};
    \draw (-2,-0.75)   node[font=\small]           {$(-++)$};
    \draw (-2,0)  node[fill=white,font=\tiny] {$(0++)$};
    \draw (-2,0.75)    node[font=\small]           {$(+++)$};
    \draw (-1.4,1.4) node[fill=white,font=\tiny] {$(+0+)$};
    \draw (0,1.5)  node[font=\small]           {$(+-+)$};
    \draw (1.4,1.4)      node[fill=white,font=\tiny] {$(+-0)$};
    \draw (2,0.75)   node[font=\small]           {$(+--)$};
    \draw (2,0)  node[fill=white,font=\tiny] {$(0--)$};
    \draw (2,-0.75)     node[font=\small]           {$(---)$};
    \draw (1.4,-1.4)   node[fill=white,font=\tiny] {$(-0-)$};
\end{tikzpicture}
\caption{The sign sequences of the faces of the hyperplane arrangement in
$\mathbb R^2$ consisting of three distinct lines.\label{figure1}}
\end{figure}
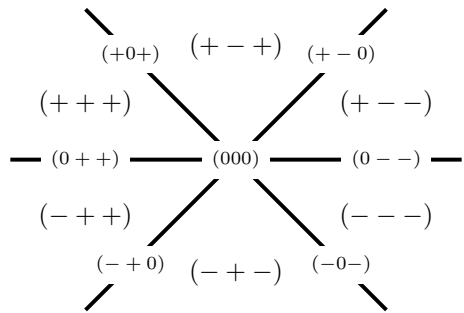

For $F, G \in \FF$ the product $FG$ is
the face of $\AA$ with sign sequence
\begin{gather*}
\sigma_H(FG) =
\begin{cases}
\sigma_H(F), & \text{if } \sigma_H(F) \neq 0, \\
\sigma_H(G), & \text{if } \sigma_H(F) = 0.
\end{cases}
\end{gather*}

\begin{Rmk} There is a nice geometric interpretation of this
product. The face $FG$ is the face that one enters by moving a \emph{small}
positive distance along any straight line from $F$ to $G$.
\end{Rmk}

The following proposition is obvious from the definition of the product in a hyperplane monoid. It justifies our use of the term geometric left regular band and the notation $\{0,+,-\}$ for $L$.
\begin{Prop}
 \label{supp properties}
Let $\AA$ be a hyperplane arrangement with $n$ hyperplanes. Then the hyperplane face monoid of $\AA$ is a submonoid of $\{0,+,-\}^n$.
\end{Prop}

Therefore every hyperplane face monoid is in $\qv(L)$. Since the hyperplane arrangement consisting of the origin in $\mathbb R^{1}$ has face monoid $L$, it follows that the quasivariety generated by all hyperplane monoids is the same as the quasivariety generated by $L$ and that a finite semigroup embeds in a hyperplane face monoid if and only if it belongs to $\qv(L)$.

To explain the complex hyperplane analogue, it is convenient to give an alternative description of the hyperplane face monoid.  Let $\sgn\colon \mathbb R\to L$ be the mapping defined by
\[\sgn(x) = \begin{cases} +, & \text{if}\ x>0\\ 0, & \text{if}\ x=0,\\ -, & \text{if}\ x<0.\end{cases}\]  Suppose that the hyperplane arrangement $\mathcal A$ in $\mathbb R^d$ is defined by the linear forms $f_1=0,\ldots, f_n=0$. Define a mapping $\sigma\colon \mathbb R^d\to L^n$ by \[\sigma(\vec x) = (\sgn(f_1(\vec x)),\ldots,\sgn(f_n(\vec x))).\]  It is not hard to see that $\sigma(\vec x)$ is the sign sequence of the face of the hyperplane arrangement containing $\vec x$ and hence $\mathcal F=\sigma(\mathbb R^d)$.

Let us define an analogue of $\sgn$ for the complex numbers.  Let $Z$ be the left regular band $\{0,+,-,i,j\}$ in which $\{i,j\}$ is a left zero semigroup, $\{0,+,-\}$ is a copy of $L$ and $L$ fixes $i,j$ on both sides.  Note that $ZL\cong \{0,+,-,i\}$.  Conversely, we can separate points of $Z$ into $ZL$ by on the one hand factoring out the ideal $\{i,j\}$ and on the other mapping $\{0,+,-\}$ to $0$ and $i,j$ to $+,-$ (respectively).  Thus $\qv(ZL) = \qv(Z)$.

Now define $\psi\colon \mathbb C\to Z$ as follows. Here $\Im (z)$ is the imaginary part of the complex number $z$.
\[\psi(z) = \begin{cases}i, & \text{if}\ \Im (z)>0\\ j, &\text{if}\ \Im (z)<0\\ \sgn(z), & \text{if}\ z\in \mathbb R\end{cases}\]

If $\mathcal A$ is a complex hyperplane arrangement in $\mathbb C^d$ given by complex linear forms $f_1=0,\ldots, f_n=0$, then we can define a mapping $\Psi\colon \mathbb C^d\to Z^n$ by \[\Psi(\vec z) = (\psi(f_1(\vec z),\ldots,\psi(f_n(\vec z))).\]  One can show that $\mathcal F=\Psi(\mathbb C)$ is a submonoid of $Z^n$ called the hyperplane face monoid of $\mathcal A$.  See~\cite{bjorner2} for details.  The hyperplane face monoid of the arrangement in $\mathbb C^n$ given by $z_i=0$, for $i=1,\ldots,n$, has face monoid $Z^n$ and hence a finite left regular band embeds in a complex hyperplane face monoid if and only if it belongs to $\qv(ZL)$.

\section{A polynomial time description of embeddability in hyperplane face monoids}
Every left regular band is a poset via its $\mathscr R$-order.  In other words, if $S$ is a left regular band and we define $x\leq y$ by $yx=x$, then $(S,\leq)$ is a poset.  The \emph{connected components} of a poset $(P,\leq)$ are the equivalence classes of the equivalence relation $\sim$ generated by $\leq$.  Note that $a\sim b$ if and only if we can find a zig-zag sequence $a=a_0\leq a_1\geq a_2\leq \cdots \geq a_n=b$.  The connected components of $P$ are precisely the connected components of the Hasse diagram of $P$, or equivalently of the order complex of $P$. Recall that the order complex of a poset $P$ is the simplicial complex whose simplices are the chains in $P$. For convenience, we write $a\sim_n b$ if there is a zig-zag path from $a$ to $b$, as above, of length $n$ (note that we allow equalities and so $\sim_{n}\subseteq \sim_{n+1}$).

\begin{Prop}\label{leftzeroimage}
Let $S$ be a left regular band.  Then the connected component relation on $S$ is the least congruence on $S$ whose corresponding quotient is a left zero semigroup.
\end{Prop}
\begin{proof}
A homomorphism of left regular bands is order-preserving and hence maps connected components into connected components.  As the connected components of a left zero semigroup are singletons, it follows that the connected component relation $\sim$ is contained in the kernel of any homomorphism to a left zero semigroup.  Since $\leq$ is compatible with left multiplication, it is immediate that $\sim$ is compatible with left multiplication.  On the other hand, since $x\geq xy$, it is immediate that $x\sim xy$ and so $\sim$ is a congruence and $S/{\sim}$ is a left zero semigroup.
\end{proof}

If $S$ is a left regular band and $a\mathrel{\mathscr L} b$ (i.e., $Sa=Sab$), then we define \[S_{a,b} = \{s\in S\mid sa=sb\}.\]  Notice that $S_{a,b}$ is a left ideal of $S$ containing $a,b$. Let us say that a left regular band $S$ satisfies condition (CC) if whenever $a\mathrel{\mathscr L} b$ and $a,b$ are in the same connected component of $S_{a,b}$, then we have $a=b$.  Let $Q_n$ be the quasi-identity that states that $a\mathrel{\mathscr L} b$ and $a\sim_n b$ in $S_{a,b}$ implies $a=b$. For example, $Q_3$ is the quasi-identity:
\begin{align*}
(ab=a) &\wedge (ba=b) \wedge (a_{1}a=a_{1}b) \wedge(a_{2}a=a_{2}b) \wedge (a_{3}a=a_{3}b)\\  &\wedge (a_{1}a=a) \wedge (a_{1}a_{2}=a_{2}) \wedge (a_{3}a_{2}=a_{2}) \wedge (a_{3}b=b)\implies a=b.
\end{align*}

Observe that $S$ satisfies (CC) if and only if it satisfies $Q_n$ for all $n$. Hence the class of rectangular bands satisfying (CC) is a quasivariety. Actually we can bound $n$ by, say, $|S|+2$ because we can remove repetitions in the middle of our zig-zag (we may need a repetition at the beginning and at the end to make the zig-zag start by going up and end by going down).  Thus (CC) can be checked in polynomial time in $|S|$.  Alternatively, for each $a\mathrel{\mathscr L} b$, we can compute the comparability graph of the poset $S_{a,b}$ in polynomial time and then do a depth-first search to check if $b$ can be reached from $a$ by a zig-zag path.

\begin{Thm}\label{firstmain}
A left regular band $S$ belongs to $\qv(L)$ if and only if it satisfies condition (CC).
\end{Thm}
\begin{proof}
Let $Q$ be the quasivariety of left regular bands satisfying (CC).  Clearly $L\in Q$ so it remains to show $Q\subseteq \qv(L)$.  Since a quasi-identity involves only finitely many variables, a semigroup $S$ belongs to a quasivariety if and only if all its finitely generated subsemigroups belongs do.  Thus we may assume without loss of generality that $S$ is finitely generated and hence finite (as the variety of left regular bands is locally finite). We proceed by induction on $|S|$ with the cases $|S|=0,1$ being clear.  Every semilattice embeds in a direct power of the two-element semilattice and so $\qv(L)$ contains all semilattices.  Recall that the relation $\mathrel{\mathscr L}$ is a congruence on any left regular band and the quotient $S/{\mathrel{\mathscr L}}$ is a semilattice by a well known result of Clifford. Thus $S/{\mathrel{\mathscr L}}\in \qv (L)$ and so we can separate any two elements of $S$ that are not $\mathrel{\mathscr L}$-related by homomorphism into $L$.

Suppose now that $a\mathrel{\mathscr L} b$.  Note that in a left regular band, every left identity is also an identity.  Suppose that $c\in S$ is not an identity and that $ca\neq cb$.  Then we have a homomorphism $\varphi\colon S\to cS$ given by $\varphi(s)=cs$ (using that $csct=cst$) and $\varphi$ separates $a,b$.  As $cS$ is a proper subsemigroup of $S$, it follows by induction that there is a homomorphism $\psi\colon cS\to L$ separating $ca,cb$.  Thus $\psi\varphi$ separates $a,b$.  Thus we may assume that $S\setminus S_{a,b}$ is either empty or consists of only the identity.  The canonical homomorphism $S\to (S_{a,b}/{\sim})\cup \{1\}$ separates $a,b$ by definition of (CC).  As it is well known and easy to see that any left zero semigroup with adjoined identity belongs to $\qv(L)$, this completes the proof in light of Proposition~\ref{leftzeroimage}.
\end{proof}

\begin{Cor}
A finite left regular band embeds in a hyperplane monoid if and only if it satisfies condition (CC).
\end{Cor}

Suppose now that $S$ is a left regular band and $a,b\in S$ with $a\mathrel{\mathscr L} b$.  Define $S'_{a,b} =\{s\in S_{a,b}\mid a\in Ss\}$.  Note that $S'_{a,b}$ is a subsemigroup of $S_{a,b}$ containing $a,b$.  Let us say that $S$ satisfies (CC') if whenever $a\mathrel{\mathscr L} b$ and $a,b$ are in the same connected component of $S'_{a,b}$ one has $a=b$.  The left regular bands satisfying (CC') form a quasivariety.  It is defined by the quasi-identities $Q_n'$ stating that $a\mathrel{\mathscr L} b$ and $a\sim_n b$ in $S'_{a,b}$ implies $a=b$.  Again, one can check in polynomial time whether $S$ satisfies (CC').

\begin{Thm}
A left regular band belongs to $\qv(ZL)$ if and only if it satisfies condition (CC').
\end{Thm}
\begin{proof}
The left regular band $ZL$ clearly satisfies (CC').  By the argument in the proof of Theorem~\ref{firstmain}, it suffices to show that if $S$ is a finite left regular band satisfying (CC'), then $S\in \qv(ZL)$, which we proceed to do by induction on $|S|$.  The base cases $|S|=0,1$ are trivial.  As before, since semilattices belong to $\qv(ZL)$, it follows that if $a,b\in S$ are not $\mathrel{\mathscr L}$-equivalent, then we can separate them by homomorphisms into $ZL$.  Suppose next that $a\mathrel{\mathscr L} b$.  If there exists $c\in S$ a non-identity element with $ca\neq cb$, then as before we can use the homomorphism $S\to cS$ to separate $a$ and $b$ and then proceed by induction using that $|cS|<|S|$.  Thus we may assume that $S_{a,b}$ contains all non-identity elements of $S$.  Let $I$ be the ideal of elements $x\in S$ with $a\notin Sx$ and let $T$ be a left zero semigroup with elements the connected components of $S'_{a,b}$.  Then we can define a homomorphism $S\to \{1\}\cup T\cup \{0\}$ by sending $I$ to $0$, the identity of $S$ (if there is one) to $1$ and sending each element of $S_{a,b}'$ to its connected component.  This gives a map from $S$ into an element of $\qv(ZL)$ separating $a,b$.  This completes the proof.
\end{proof}

\begin{Cor}
A finite left regular band embeds in a complex hyperplane semigroup if and only if it satisfies condition (CC').
\end{Cor}

\section{On the non-existence of a finite basis of quasi-identities}
Observe that if a quasivariety is finitely based (that is, has a finite basis of quasi-identities), then it can be defined by a set of quasi-identities over a finite alphabet.  Clearly, if a quasivariety $Q$ can be defined by quasi-identities in $k$ variables, then a semigroup $S$ belongs to $Q$ if and only if each $k$-generated subsemigroup of $S$ belongs to $Q$.  Thus to show that $Q$ cannot be defined by quasi-identities over a finite alphabet (and hence is not finitely based), it suffices to construct for each $k\geq 0$ a semigroup $S_k$ such that $S_k\notin Q$ but all $k$-generated subsemigroups of $S_k$ belong to $Q$.

In this section we construct a sequence $B_n$, for $n\geq 3$, of left regular bands such that $|B_n|\to \infty$, $B_n\notin \qv(ZL)$ and every proper subsemigroup of $B_n$ belongs to $\qv(L)$.  As the variety of left regular bands is locally finite, it follows that the minimal number of generators of $B_n$ goes to infinity as $n\to \infty$. The argument of the previous paragraph then applies to yield that $\qv(L)$ and $\qv(ZL)$ cannot be defined by quasi-identities in finitely many variables and hence cannot be finitely based.  Our construction is based on hyperplane face monoids.

Let $F_n$ be the hyperplane face monoid associated to the $n$ lines through the origin of $\mathbb R^2$ at angles $2k\pi/n$ to the $x$-axis with $k=0,1,\ldots,n-1$, for $n\geq 2$.  Figure~\ref{figure1} shows $F_3$.  Let $C_1,C_2,\ldots, C_{2n}$ be the chambers visited in counter-clockwise order starting from the $x$-axis and let $r_1,r_2,\ldots,r_{2n}$ be the rays in counter-clockwise order starting from the positive $x$-axis.  See Figure~\ref{figure2} for the case $n=3$.
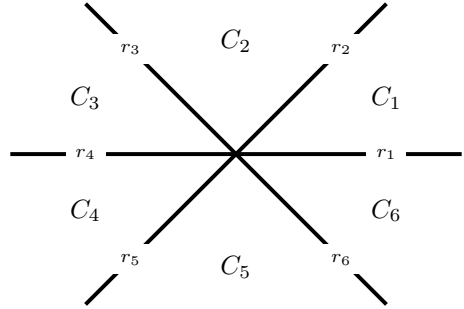
\begin{figure}[tb]
\centering
\begin{tikzpicture}
    \draw[line width=1.5pt] (-2,-2) -- (2,2);
    \draw[line width=1.5pt] (2,-2) -- (-2,2);
    \draw[line width=1.5pt] (-3,0) -- (3,0);
    \draw (0,-1.5)    node[font=\small]           {$C_5$};
    \draw (-1.4,-1.4)       node[fill=white,font=\tiny] {$r_5$};
    \draw (-2,-0.75)   node[font=\small]           {$C_4$};
    \draw (-2,0)  node[fill=white,font=\tiny] {$r_4$};
    \draw (-2,0.75)    node[font=\small]           {$C_3$};
    \draw (-1.4,1.4) node[fill=white,font=\tiny] {$r_3$};
    \draw (0,1.5)  node[font=\small]           {$C_2$};
    \draw (1.4,1.4)      node[fill=white,font=\tiny] {$r_2$};
    \draw (2,0.75)   node[font=\small]           {$C_1$};
    \draw (2,0)  node[fill=white,font=\tiny] {$r_1$};
    \draw (2,-0.75)     node[font=\small]           {$C_6$};
    \draw (1.4,-1.4)   node[fill=white,font=\tiny] {$r_6$};
\end{tikzpicture}
\caption{Three lines in the plane.\label{figure2}}
\end{figure}
Assume now that $n\geq 3$.  Denote by $F_n'$ the subsemigroup obtained from $F_n$ by removing the origin and the $\mathscr L$-class corresponding to the $x$-axis, i.e., $r_1$ and $r_{n+1}$.  Then $C_n,C_{n+1}$ are separated only by the $x$-axis and hence are on the same side of each line remaining in $F_n'$.  It follows that $xC_n=xC_{n+1}$ for all $x\in F_n'$.  Also $C_n,C_{n+1}$ are left zeroes.  Therefore the equivalence relation on $F_n'$ whose only non-singleton block is $\{C_n,C_{n+1}\}$ is a congruence.  Let $B_n$ be the corresponding quotient and denote by $C$ the equivalence class $\{C_n,C_{n+1}\}$.  See Figure~\ref{figure3} for the case $n=3$.
\begin{figure}[tb]
\centering
\begin{tikzpicture}
    \draw[line width=1.5pt] (-2,-2) -- (2,2);
    \draw[line width=1.5pt] (2,-2) -- (-2,2);
    \draw[dashed, thick] (0,0) -- (3,0);
    \draw (0,-1.5)    node[font=\small]           {$C_5$};
    \draw (-1.4,-1.4)       node[fill=white,font=\tiny] {$r_5$};
    \draw (-2,0)  node[fill=white,font=\small] {$C$};
    \draw (-1.4,1.4) node[fill=white,font=\tiny] {$r_3$};
    \draw (0,1.5)  node[font=\small]           {$C_2$};
    \draw (1.4,1.4)      node[fill=white,font=\tiny] {$r_2$};
    \draw (2,0.75)   node[font=\small]           {$C_1$};
    \draw (2,0)  node[fill=white,font=\tiny] {$r_1$};
    \draw (2,-0.75)     node[font=\small]           {$C_6$};
    \draw (1.4,-1.4)   node[fill=white,font=\tiny] {$r_6$};
\end{tikzpicture}
\caption{The semigroup $B_3$.\label{figure3}}
\end{figure}
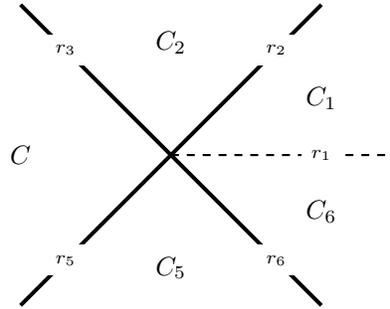

Let us begin by proving that $B_n$ does not belong to $\qv{(ZL)}$.  First observe that since $C_1,C_{2n}$ are only separated by the $x$-axis, whose faces are removed to form $F_n'$, it follows that $xC_1=xC_{2n}$ for all $x\in F_n'$ and hence the same is true in $B_n$.  Thus $B_n=(B_n)_{C_1,C_{2n}}$.  It suffices now to compute the connected components of $B_n$.  The Hasse diagram of $B_n$ is the zig-zag path \[C_1< r_2> C_2< r_3\cdots < r_n> C< r_{n+2}> C_{n+2}< r_{n+3} \cdots <r_{2n}> C_{2n}\] from $C_1$ to $C_{2n}$; see Figure~\ref{figure4} for the case $n=3$.  It follows that $B_n$ does not satisfy (CC').

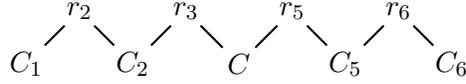
\begin{figure}[tb]
\centering
\begin{tikzpicture}
\node (C) at (0,0) {$C$};
\node [above left of =C] (r3) {$r_3$};
\node [above right of =C] (r5) {$r_5$};
\node [below left of =r3] (C2) {$C_2$};
\node [below right of =r5] (C5) {$C_5$};
\node [above left of =C2] (r2) {$r_2$};
\node [above right of =C5] (r6) {$r_6$};
\node [below left of =r2] (C1) {$C_1$};
\node [below right of =r6] (C6) {$C_6$};
\draw [thick, shorten <=-2pt, shorten >=-2pt] (C1) to (r2);
\draw [thick, shorten <=-2pt, shorten >=-2pt] (C2) to (r2);
\draw [thick, shorten <=-2pt, shorten >=-2pt] (C2) to (r3);
\draw [thick, shorten <=-2pt, shorten >=-2pt] (C) to (r3);
\draw [thick, shorten <=-2pt, shorten >=-2pt] (C) to (r5);
\draw [thick, shorten <=-2pt, shorten >=-2pt] (C5) to (r5);
\draw [thick, shorten <=-2pt, shorten >=-2pt] (C5) to (r6);
\draw [thick, shorten <=-2pt, shorten >=-2pt] (C6) to (r6);
\end{tikzpicture}
\caption{Hasse diagram of $B_3$.\label{figure4}}
\end{figure}

Identifying $C_1$ and $C_{2n}$ is a congruence on $B_n$ (because $C_1,C_{2n}$ lie on the same side of every hyperplane except the $x$-axis, which is removed in forming $F_n'$, and they are left zeroes).  The quotient is isomorphic to $F_{n-1}$ with the origin removed.  Thus we can separate all pairs of elements of $B_n$ except $\{C_1,C_{2n}\}$ by a homomorphism into a hyperplane face monoid.  In particular, any subsemigroup of $B_n$ that does not contain both $C_1$ and $C_{2n}$ belongs to $\qv(L)$.  On the other hand, because the Hasse diagram of $B_n$ is a line from $C_1$ to $C_{2n}$ if $S$ is any proper subsemigroup of $B_n$ containing $\{C_1,C_{2n}\}$, then $C_1$ and $C_{2n}$ are in different connected components of $S$ and hence $S\to S/{\sim}$ is a homomorphism from $S$ to a left zero semigroup (and hence a member of $\qv(L)$) separating $C_1$ and $C_{2n}$.  This completes the proof that any proper subsemigroup of $B_n$ belongs to $\qv(L)$.  We have thus proven the following.

\begin{Thm} \label{nofinbasis}
If $Q$ is a quasivariety generated by a finite semigroup and  $\qv(L) \subseteq Q \subseteq \qv(ZL)$ then $Q$ has no finite basis of quasi-identities.  In fact, $Q$ cannot be defined by quasi-identities over a finite alphabet.
\end{Thm}

As a corollary, we obtain $\qv(L)$ is a minimal non-finitely based quasivariety.

\begin{Cor}
$\qv(L)$ is a minimal non-finitely based quasivariety.
\end{Cor}

\section{Open problems}

In Section~\ref{history} we mentioned two open problems. We describe them in detail here.

\begin{Problem}
Is it true that the quasivarieties $\qv(L)$ or $\qv(ZL)$ have no independent basis for their quasi-identities?
\end{Problem}

We remarked in Section~\ref{history} that Sapir proved~\cite{Sapir3} that the semigroup $C_{3,1}$ has no independent basis of quasi-identities. It is therefore reasonable to ask if the same is true for $\qv(L)$ and $\qv(ZL)$.

\begin{Problem}
Is it true that any left or right regular band (band, completely regular semigroup, finite semigroup) that contains either $L$ or its dual $R$ as a subsemigroup has no finite basis for its quasi-identities? That is, is $L$ SNFQB for left regular bands (bands, completely regular semigroups, finite semigroups)?
\end{Problem}

Of course a positive solution to this problem would be a major contribution to the decidability of the finire basis problem for quasi-identities of finite semigroups. As mentioned in section~\ref{history}, no semigroup with index more than $2$ has a finite basis of quasi-identities~\cite{JV}. Sapir's thesis~\cite{SapirThesis} has examples of bands in the variety $[\![x^{2}=x, xyz =xyzxz]\!]$ that are SNQFB for the class of all finite semigroups, but none of these are LRBs.

Thus a positive solution to this problem for completely regular semigroups, for example,  would complete the decidability question for semigroups of index $1$. As far as we know, no semigroup of size greater than $2$ and not covered by the results of~\cite{JV} or~\cite{Gerhard} has been proven to have a finite basis for its quasi-identities. In particular, there is no non-normal band that is known to have a finite basis for its quasi-identities (to the best of our knowledge).

\section*{Acknowledgements}

The authors would like to thank Igor Dolinka for helpful correspondence on questions related to this paper and Mark Sapir for explaining in detail the results of his thesis~\cite{SapirThesis} that impinge upon this paper.
\def\malce{\mathbin{\hbox{$\bigcirc$\rlap{\kern-7.75pt\raise0,50pt\hbox{${\tt
  m}$}}}}}\def\cprime{$'$} \def\cprime{$'$} \def\cprime{$'$} \def\cprime{$'$}
  \def\cprime{$'$} \def\cprime{$'$} \def\cprime{$'$} \def\cprime{$'$}
  \def\cprime{$'$}

\end{document}